\documentclass[11pt]{amsart}
\usepackage{amsxtra,amssymb,amsmath,amsthm,amsbsy}
\usepackage[all]{xy}

\theoremstyle{plain}

\numberwithin{equation}{section}

\newtheorem{theorem}{Theorem}[section]

\newtheorem{definition-lemma}[theorem]{Definition-Lemma}
\newtheorem{proposition}[theorem]{Proposition}

\newtheorem{example}[theorem]{Example}
\newtheorem{remark}[theorem]{Remark}

\newcommand{\id}         {{\mathrm {Id}}}
\newcommand{\image}      {{\mathrm {Im}}}
\newcommand{\Ker}        {{\mathrm {ker}}}

\newcommand{\pr}         {{\mathrm{pr}}}

\def\id{{\rm id}}

\def\pr{{\rm pr}}

\def\toto{\rightrightarrows}

\def\<{\langle}
\def\>{\rangle}

\newcommand{\SP} [1]     {{\left\langle {{#1}} \right\rangle}}

\newcommand{\Grd}        {\mathcal{G}}
\newcommand{\grd}         {\mathcal{G}}
\newcommand{\sour}        {\mathsf{s}}
\newcommand{\tar}         {{\mathsf{t}}}
\newcommand{\inv}         {{\mathsf{inv}}}

\newcommand{\Cour}[1]      {[\![#1]\!]}

\newcommand{\TM}        {\mathbb{T}M}

\newcommand{\Lie}        {\mathcal L}

\begin{document}
\title[]
{Multisymplectic geometry and Lie groupoids}

\author[]{Henrique Bursztyn, Alejandro Cabrera and David Iglesias}

\address{IMPA,
Estrada Dona Castorina 110, Rio de Janeiro, 22460-320, Brasil }
\email{henrique@impa.br}

\address{Departamento de Matem\'atica Aplicada - IM, UFRJ. Av. Athos da Silveira Ramos 149 (CT bloco
C) Cidade Universitaria 21941-909 - Rio de Janeiro, RJ - Brasil -
Caixa-postal: 68530} \email{acabrera@labma.ufrj.br}

\address{Departamento de Matem\'atica Fundamental, Facultad de Matem\'aticas,
Universidad de la Laguna, Spain} \email{diglesia@ull.es}

\date{}

\maketitle

\begin{abstract}
We study higher-degree generalizations of symplectic groupoids,
referred to as {\em multisymplectic groupoids}. Recalling that
Poisson structures may be viewed as infinitesimal counterparts of
symplectic groupoids, we describe ``higher'' versions of Poisson
structures by identifying the infinitesimal counterparts of
multisymplectic groupoids. Some basic examples and features are
discussed.
\end{abstract}

\begin{center} {\it In memory of Jerry Marsden}
\end{center}

\tableofcontents

\section{Introduction}

Multisymplectic structures are higher-degree analogs of symplectic
forms which arise in the geometric formulation of classical field
theory much in the same way that symplectic structures emerge in the
hamiltonian description of classical mechanics, see
\cite{Got,Hel,KS} and references therein. This symplectic approach
to field theory was explored in a number of Marsden's publications,
which treated (as it was typical in Marsden's work) theoretical as
well as applied aspects of the subject, see e.g.
\cite{GIMM,GIM,Mar1,Mar2}. Multisymplectic geometry (as in
\cite{CID2,CID}) also arises in other settings, such as the study of
homotopical structures \cite{CR}, categorified symplectic geometry
\cite{BHR}, and geometries defined by closed forms \cite{MS}.

Poisson structures are generalizations of symplectic structures
which are central to geometric mechanics\footnote{E.g., in the
description of the interplay between hamiltonian dynamics and
symmetries \cite{MR}, and in the transition from classical to
quantum mechanics \cite{CW}.} and permeate Marsden's work. A natural
problem in multisymplectic geometry is the identification of
``higher'' analogs of Poisson structures bearing a relation to
multisymplectic forms that extends the way Poisson geometry
generalizes symplectic geometry. In this note we discuss one
possible approach to tackle this issue.

Our viewpoint relies on the relationship between Poisson geometry
and objects known as {\em symplectic groupoids} \cite{CDW,We87}.
This relationship is part of a generalized Lie theory in which
Poisson structures arise as infinitesimal, or linearized,
counterparts of symplectic groupoids, in a way analogous to how Lie
algebras correspond to Lie groups. In order to find higher analogs
of Poisson structures the route we take is to first consider
higher-degree versions of symplectic groupoids, referred to as {\em
multisymplectic groupoids}, and then to identify the geometric
objects arising as their infinitesimal counterparts. Recalling that
symplectic groupoids are Lie groupoids equipped with a symplectic
structure that is compatible with the groupoid multiplication, in
the sense that the symplectic form is {\em multiplicative} (see
\eqref{eq:multip} below), multisymplectic groupoids are defined
analogously, as Lie groupoids endowed with a multiplicative
multisymplectic structure. Our identification of the infinitesimal
objects corresponding to multisymplectic groupoids builds on the
infinitesimal description of general multiplicative differential
forms obtained in \cite{AC,bc}.

For a manifold $M$, our ``higher-degree'' analogs of Poisson
structures can be conveniently expressed (in the spirit of Dirac
geometry \cite{courant}) in terms of subbundles
\begin{equation}\label{eq:L}
L\subset TM\oplus \wedge^k T^*M
\end{equation}
satisfying suitable properties, including an involutivity condition
with respect to the ``higher'' Courant-Dorfman bracket on the space
of sections of $TM\oplus \wedge^k T^*M$ (see e.g.
\cite[Sec.~2]{Hi}). Related geometric objects have been recently
considered in the study of higher analogs of Dirac structures in
\cite{Zambon} (see also \cite{VYM}). But, as it turns out, the
higher Poisson structures that arise from multisymplectic groupoids
are not particular cases of the higher Dirac structures of
\cite{Zambon} (for example, comparing with \cite[Def.~3.1]{Zambon},
the higher Poisson structures \eqref{eq:L} considered here are not
necessarily lagrangian subbundles, though always isotropic). An
alternative characterization of these objects, more in the spirit of
the bivector-field description of Poisson structures, is presented
in Prop.~\ref{prop:D}.

Another perspective on higher Poisson structures relies on the view
of Poisson structures as Lie brackets on the space of smooth
functions of a manifold. A natural issue in this context is finding
an appropriate extension of the Poisson bracket defined by a
symplectic form (see \eqref{eq:poissons}) to multisymplectic
manifolds. This problem involves notorious difficulties and much
work has been done on it, see e.g. \cite{Forger1,Kan,CR}. The
approach to higher Poisson structures in this note follows a
different path and does not address any of the issues involved in
the algebraic study of higher Lie-type brackets.

The paper is structured as follows. We review Poisson structures and
their connection with symplectic groupoids in Section
\ref{sec:poisson}. In Section \ref{sec:multi} we recall the basics
of multisymplectic forms. The main results are presented in Section
\ref{sec:multigrp}, in which we introduce multisymplectic groupoids
and identify their infinitesimal counterparts. In Section
\ref{sec:higherc} we give different descriptions of these objects
and explain some of their properties, while examples are discussed
in Section \ref{sec:examples}.

As one should expect, higher Poisson structures naturally arise in
connection with symmetries in multisymplectic geometry. This aspect
of the subject is not treated here, though we hope to
explore it, as well as its relations with field theory, in future
work. Parallel ideas to those in this note can be also carried out
in the context of polysymplectic geometry, see \cite{Nicolas}.

\smallskip

\noindent {\bf Acknowledgments}: H.B. and A.C. thank the organizers
of the {\em Focus Program on Geometry, Mechanics and Dynamics: the
Legacy of Jerry Marsden}, held at the Fields Institute in July of
2012, for their hospitality during the program, as well as MITACS
for travel support (for which we also thank J. Koiller). H. B. was
partly funded by CNPq and Faperj. D. I. thanks MICINN (Spain) for a
``Ram\'on y Cajal" research contract; he is partially supported by
MICINN grants MTM2009-13383 and MTM2009-08166-E and Canary Islands
government project SOLSUB200801000238. We have benefited from many
stimulating conversations with M. Forger, J. C. Marrero, N.
Martinez, C. Rogers and M. Zambon. We also thank the referees for
several useful comments that improved the presentation of this note.

\section{Poisson structures and symplectic groupoids}
\label{sec:poisson}

We start by recalling a few different viewpoints to Poisson
structures.

A {\it Poisson structure} on a smooth manifold $M$ is Lie bracket
$\{\cdot,\cdot\}$ on $C^\infty(M)$ which is compatible with the
pointwise product of functions via the Leibniz rule:
\begin{equation}\label{eq:leibniz}
\{f,gh\}=\{f,g\}h + \{f,h\}g,\;\;\; f,g,h \in C^\infty(M).
\end{equation}
The Leibniz condition \eqref{eq:leibniz} implies that
$\{\cdot,\cdot\}$ is necessarily defined by a bivector field $\pi
\in \Gamma(\wedge^2 TM)$ via
$$
\pi(df,dg) = \{f,g\},\;\;\; f,g \in C^\infty(M).
$$
This leads to the alternative description of Poisson structures on
$M$ as bivector fields $\pi \in \Gamma(\wedge^2 TM)$ satisfying
$[\pi,\pi]=0$, where $[\cdot,\cdot]$ is the Schouten-Nijenhuis
bracket on multivector fields. (The vanishing of $[\pi,\pi]$
accounts for the Jacobi identity of $\{\cdot,\cdot\}$.) We denote
Poisson manifolds by either $(M,\pi)$ or $(M,\{\cdot,\cdot\})$.

Symplectic manifolds are naturally equipped with Poisson structures.
Given a symplectic manifold $(M,\omega)$, and denoting by $X_f$ the
hamiltonian vector field associated with $f\in C^\infty(M)$ via
\begin{equation}\label{eq:hamv}
i_{X_f}\omega = df,
\end{equation}
the Poisson bracket on $M$ is given by
\begin{equation}\label{eq:poissons}
\{f,g\} = \omega(X_g,X_f).
\end{equation}

A more recent perspective on Poisson structures, which is the
guiding principle of this note, relies on another type of connection
between Poisson structures and symplectic manifolds. It is based on
the fact that Poisson geometry fits into a generalized Lie theory,
naturally expressed in terms of Lie algebroids and groupoids, see
e.g. \cite{CDW}. In this context, Poisson manifolds are seen as
infinitesimal counterparts of global objects called {\it symplectic
groupoids} \cite{We87}, analogously to how Lie algebras are regarded
as infinitesimal versions of Lie groups. We will briefly recall the
main aspects of the theory.

Let $\grd \toto M$ be a Lie groupoid (the reader can find
definitions and further details in \cite{CW}). We use the following
notation for its structure maps: $\sour$, $\tar : \grd \to M$ for
the source, target maps, $m: \grd {_\sour\times_\tar} \grd \to \grd$
for the multiplication map\footnote{Here the fibred product $\grd
{_\sour\times_\tar}\grd = \{(g,h) \in \grd\times \grd\,|\,
\sour(g)=\tar(h)\}$ represents the space of composable arrows.},
$\epsilon: M\hookrightarrow \grd$ for the unit map, and
$\inv:\grd\to \grd$ for the groupoid inversion. We will often
identify $M$ with its image under $\epsilon$ (the submanifold of
$\grd$ of identity arrows).

A differential form $\omega \in \Omega^r(\grd)$ is called {\it
multiplicative} if it satisfies
\begin{equation}\label{eq:multip}
m^*\omega = \pr_1^*\omega + \pr_2^*\omega,
\end{equation}
where $\pr_i: \grd {_\sour\times_{\tar}}\grd \to \grd$, $i=1,2$, is
the natural projection onto the $i$-th factor\footnote{For a
function $f\in \Omega^0(\grd)=C^\infty(\grd)$, condition
\eqref{eq:multip} becomes $f(gh)=f(g)+f(h)$, i.e., it says that $f$
is a groupoid morphism into $\mathbb{R}$ (viewed as an abelian
group).}. A {\em symplectic groupoid} is a Lie groupoid $\grd \toto
M$ equipped with a multiplicative symplectic form $\omega \in
\Omega^2(\grd)$. In this case, condition \eqref{eq:multip} is
equivalent to the graph of the multiplication map $m$ being a
lagrangian submanifold of $\grd\times \grd \times \overline{\grd}$,
where $\overline{\grd}$ is equipped with the opposite symplectic
form $-\omega$. Symplectic groupoids first arose in symplectic
geometry in the context of quantization (see e.g.
\cite[Sec.~8.3]{BW}) but turn out to provide a convenient setting
for the study of symmetries and reduction \cite{MiWe}.

In order to explain how symplectic groupoids are related to Poisson
structures, recall that a {\it Lie algebroid} is a vector bundle
$A\to M$ equipped with a bundle map $\rho: A\to TM$, called the {\em
anchor}, and a Lie bracket $[\cdot,\cdot]$ on $\Gamma(A)$ such that
$$
[u,fv] = f[u,v] + (\Lie_{\rho(u)}f)v,
$$
for $u,v \in \Gamma(A),\; f\in C^\infty(M)$. Lie algebroids are
infinitesimal versions of Lie groupoids: for a Lie groupoid
$\grd\toto M$, its associated Lie algebroid is defined by
$A=\ker(d\sour)|_M$, with anchor map $d\tar|_A: A\to TM$ and Lie
bracket on $\Gamma(A)$ induced by the Lie bracket of right-invariant
vector fields on $\Grd$. Much of the usual theory relating Lie
algebras and Lie groups carries over to Lie algebroids and
groupoids, a notorious exception being Lie's third theorem, i.e.,
not every Lie algebroid arises as the Lie algebroid of a Lie
groupoid (see \cite{CF} for a thorough discussion of this issue).

The first indication of a connection between Poisson geometry and
Lie algebroids/groupoids is the fact that, if $(M,\pi)$ is a Poisson
manifold, then its cotangent bundle $T^*M\to M$ inherits a Lie
algebroid structure, with anchor map given by
\begin{equation}\label{eq:anchor}
\pi^\sharp: T^*M \to TM,\;\;\; \pi^\sharp(\alpha)=i_\alpha\pi,
\end{equation}
and Lie bracket on $\Gamma(T^*M)=\Omega^1(M)$ given by
\begin{equation}\label{eq:liep}
[\alpha,\beta]=\Lie_{\pi^\sharp(\alpha)}\beta -
\Lie_{\pi^\sharp(\beta)}\alpha - d(\pi(\alpha,\beta)).
\end{equation}

The precise relation between Poisson structures and symplectic
groupoids is as follows. First, given a symplectic groupoid
$(\grd\toto M,\omega)$, its space of units $M$ inherits a natural
Poisson structure $\pi$, uniquely determined by the fact that the
target map $\tar: \grd\to M$ is a Poisson map (while $\sour: \grd\to
M$ is anti-Poisson); moreover, denoting by $A$ the Lie algebroid of
$\grd$, there is a canonical identification between $A$ and the Lie
algebroid structure on $T^*M$ induced by $\pi$, explicitly given by
$$
\mu: A \stackrel{\sim}{\to} T^*M,\;\; \mu(u) = i_u\omega|_{TM}.
$$
Here we view $TM$ as a subbundle of $T\grd|_M$ via $\epsilon:
M\hookrightarrow \grd$, so that we can write
\begin{equation}\label{eq:decomp}
T\grd|_M = TM\oplus A.
\end{equation}
In other words, the Lie groupoid $\grd$  integrates the Lie
algebroid $T^*M$ defined by $\pi$.

Conversely, given a Poisson manifold $(M,\pi)$ and assuming that its
associated Lie algebroid is integrable (i.e., can be realized as the
Lie algebroid of a Lie groupoid\footnote{See e.g. \cite{We87} for a
nonintegrable example and \cite{CF2} for a discussion of
obstructions to integrability.}), then its $\sour$-simply-connected
integration $\grd\toto M$ inherits a symplectic groupoid structure.
(As shown in \cite{catfel}, one can obtain $\grd$ by means of an
infinite-dimensional Marsden-Weinstein reduction.)

The upshot of this discussion is that {\em Poisson manifolds are the
infinitesimal versions of symplectic groupoids.}

Some of the prototypical examples of symplectic groupoids are
traditional phase spaces in mechanics. For example, any cotangent
bundle $T^*Q$, equipped with its canonical symplectic form, is a
symplectic groupoid over $Q$ with respect to the groupoid structure
given by fibrewise addition of covectors; in this case, source and
target maps coincide, both being the bundle projection $T^*Q\to Q$,
and the corresponding Poisson structure on $Q$ is trivial: $\pi=0$.
A more interesting example is given by the cotangent bundle of a Lie
group $G$. In this case, besides the symplectic groupoid structure
over $G$ that we just described, $T^*G$ is also a symplectic
groupoid over $\mathfrak{g}^*$, where $\mathfrak{g}$ denotes the Lie
algebra of $G$. The groupoid structure
$$
T^*G\toto \mathfrak{g}^*
$$
is induced by the co-adjoint action of $G$ on $\mathfrak{g}^*$ (see
e.g. \cite{MiWe}); source and target maps are given by the momentum
maps for the cotangent lifts of the actions of $G$ on itself by left
and right translations, while the corresponding Poisson structure on
$\mathfrak{g}^*$ is just its natural Lie-Poisson structure. The fact
that the target map is a Poisson map may be viewed as the {\em
Lie-Poisson reduction theorem} (see e.g. \cite[Sec.~13.1]{MR}),
another one of Marsden's favorite topics. The correspondence between
Poisson structures and symplectic groupoids extends much of the
theory relating $\mathfrak{g}^*$ and $T^*G$ to more general
settings.

\section{Multisymplectic structures}\label{sec:multi}

A {\em multisymplectic structure} \cite{CID2,CID} on a manifold $M$
is a differential form $\omega\in \Omega^{k+1}(M)$ which is closed
and nondegenerate, in the sense that $i_X\omega=0$ implies that
$X=0$, for $X\in \Gamma(TM)$. Equivalently, the nondegeneracy
condition says that the bundle map
\begin{equation}\label{eq:mnondeg}
\omega^\sharp: TM \to \wedge^kT^*M,\;\;\; X\mapsto i_X\omega,
\end{equation}
is injective. As in \cite{BHR,CR}, we refer to a multisymplectic
form of degree $k+1$ as a {\em $k$-plectic} form. Hence a
$1$-plectic form $\omega$ is a usual symplectic structure, in which
case the map \eqref{eq:mnondeg} is necessarily surjective; note that
the wedge powers $\omega^r$, $r=2,\ldots,\dim(M)$, are natural
examples of higher degree multisymplectic forms. For completeness,
we briefly recall some other examples, see e.g. \cite{CID}.

For a manifold $Q$, the total space of the exterior bundle $\wedge^k
T^*Q$ carries a canonical $k$-plectic form $\omega_{can}$,
generalizing the canonical symplectic structure on $T^*Q$. Indeed,
there is a ``tautological'' $k$-form $\theta$ on $\wedge^k T^*Q$
given by
$$
\theta_{\xi}(X_1,\ldots,X_k)=\xi(dp(X_1),\ldots,dp(X_k)),
$$
where $p: \wedge^k T^*Q \to Q$ is the natural bundle projection,
$\xi \in \wedge^k T^*Q$, and $X_i$, $i=1,\ldots,k$, are tangent
vectors to $\wedge^k T^*Q$ at $\xi$. Then
\begin{equation}\label{eq:can}
\omega_{can} = d\theta
\end{equation}
is a $k$-plectic form on $\wedge^kT^*Q$. These $k$-plectic manifolds
are closely related to the multi-phase spaces in field theory (see
e.g. \cite{GIMM,Hel} and references therein).

Other examples of $k$-plectic manifolds include $(k+1)$-dimensional orientable
manifolds equipped with volume forms. An important class of
2-plectic manifolds is given by compact, semi-simple Lie groups $G$,
equipped with the Cartan 3-form $H\in \Omega^3(G)$, i.e., the bi-invariant
3-form uniquely defined by the condition $H(u,v,w)=\SP{u,[v,w]}$, where
$u$, $v$, $w\in \mathfrak{g}$ and $\SP{\cdot,\cdot}$ is the Killing form
(see e.g. \cite{BHR,CID}).
Hyper-K\"ahler manifolds are examples of 3-plectic manifolds: if
$\omega_1$, $\omega_2$, $\omega_3$ are the three K\"ahler forms on a hyper-K\"ahler
manifold $M$, then the form $\omega_1\wedge\omega_1 +
\omega_2\wedge\omega_2 + \omega_3\wedge\omega_3 \in \Omega^4(M)$ is
3-plectic \cite{CID,MS}.

In physical applications (such as quantization), an important issue
concerns the identification of an appropriate analog of the Poisson
bracket \eqref{eq:poissons} on a $k$-plectic manifold $(M,\omega)$;
there is an extensive literature on this problem, see
\cite{CID2,Forger1,Kan,CR}. As a starting point, one usually
considers forms $\alpha\in \Omega^{k-1}(M)$ for which there exists a
(necessarily unique) vector field $X_\alpha$ such that
$i_{X_\alpha}\omega = d\alpha$; such forms are called {\em
hamiltonian}. Then, on the space of hamiltonian $(k-1)$-forms, one
defines the bracket
\begin{equation}\label{eq:hpoisson}
\{\alpha,\beta\}=i_{X_\alpha}i_{X_\beta}\omega,
\end{equation}
which is a direct generalization of the Poisson bracket
\eqref{eq:poissons} when $k=1$. This skew-symmetric bracket turns
out to be well defined on the space of hamiltonian $(k-1)$-forms,
but the Jacobi identity usually fails (see e.g. \cite{CID2,CR}):
\begin{equation}\label{eq:jac}
\{\alpha,\{\beta,\gamma\}\} + \{\gamma,\{\alpha,\beta\}\} +
\{\beta,\{\gamma,\alpha\}\} = -d
i_{X_\alpha}i_{X_\beta}i_{X_\gamma}\omega.
\end{equation}

Much work has been done to deal with this ``defect'' on the
jacobiator of \eqref{eq:hpoisson}, either by forcing its elimination
or by somehow making sense of it. One approach relies on noticing
that closed $(k-1)$-forms are automatically hamiltonian, so one can
consider the quotient space of hamiltonian forms modulo closed forms
(see e.g. \cite{CID2}); the bracket \eqref{eq:hpoisson} descends to
this quotient and, since the right-hand side of \eqref{eq:jac} is
exact, the quotient inherits a genuine Lie-algebra
structure\footnote{In the case of exact $k$-plectic manifolds, a
different way to eliminate the jacobiator defect is presented in
\cite{Forger1}, based on a modification of the bracket
\eqref{eq:hpoisson} using the $k$-plectic potential.}. By using
multivector fields, one can also consider hamiltonian forms of other
degrees and show that these Lie algebras fit into larger graded Lie
algebras. A more recent approach, see \cite{BHR,CR}, shows that,
without taking quotients (so as to force the vanishing of the
jacobiator), the bracket \eqref{eq:hpoisson} on hamiltonian forms
can be naturally understood in terms of structures from homotopy
theory; namely, this bracket is part of a Lie $k$-algebra (a special
type of $L_\infty$-algebra). A missing ingredient in these
generalizations of the Poisson bracket \eqref{eq:poissons} is a
corresponding analog of the Leibniz rule \eqref{eq:leibniz}. For a
discussion in this direction, see e.g. \cite{Hra,Kan}.

Just as symplectic manifolds are particular cases of Poisson
manifolds, one could wonder about the analog of Poisson manifolds in
multisymplectic geometry. As recalled in Section \ref{sec:poisson},
the Leibniz rule is central for the general definition of a Poisson
structure. So, as indicated by the previous discussion on Poisson
brackets on $k$-plectic manifolds, it is not evident how to define
such analogs in terms of algebraic/Lie-type structures on spaces of
forms. A different, more geometric, perspective to this problem will
be discussed next.

\section{Multisymplectic groupoids and their infinitesimal versions}
\label{sec:multigrp}

We start with a straightforward generalization of symplectic
groupoids to multisymplectic geometry: A {\it multisymplectic
groupoid} is a Lie groupoid equipped with a multisymplectic form
that is multiplicative, in the sense of \eqref{eq:multip}. We will
also use the terminology {\em $k$-plectic groupoid} when the
multisymplectic form has degree $k+1$.

Recalling that Poisson structures arise as infinitesimal versions of
symplectic groupoids, as briefly explained in Section
\ref{sec:poisson}, we will now identify the infinitesimal objects
corresponding to multisymplectic groupoids.

Let $\grd \toto M$ be an $\sour$-simply-connected Lie groupoid, let
$A\to M$ be its Lie algebroid, with anchor map $\rho: A\to TM$. The
following result is established in \cite{AC,bc}: there is a 1-1
correspondence between closed, multiplicative forms $\omega \in
\Omega^{k+1}(\grd)$ and vector-bundle maps $\mu: A\to \wedge^k T^*M$
(covering the identity map on $M$) satisfying:
\begin{align}
&i_{\rho(u)}\mu(v) = -i_{\rho(v)}\mu(u),\label{eq:IM1}\\
&\mu([u,v]) = \Lie_{\rho(u)}\mu(v) - i_{\rho(v)}d (\mu(u)),
\label{eq:IM2}
\end{align}
for $u, v \in \Gamma(A)$. Such maps $\mu$ are called (closed) {\em
IM $(k+1)$-forms} (where IM stands for {\em infinitesimally
multiplicative}). Using \eqref{eq:decomp}, one can write the
explicit relation between $\omega$ and $\mu$ as
\begin{equation}\label{eq:rel}
i_{X_k} \ldots i_{X_1}\mu_x(u) = \omega_x(u,X_1,\ldots,X_k),
\end{equation}
for $u \in A|_x$ and $X_i\in TM|_x$, $x\in M$.

We now discuss a slight refinement of this result taking into
account the nondegeneracy condition of multisymplectic forms. We
will need a few properties of multiplicative forms on Lie groupoids,
all of which follow from  \eqref{eq:multip}. If $\omega$ is a
multiplicative form on $\grd$, then the following holds:
\begin{equation}\label{eq:propm}
\epsilon^*\omega=0,\qquad \inv^*\omega=-\omega,
\end{equation}
and
\begin{equation}\label{eq:mu}
i_{u^r}\omega = \tar^* \mu(u),\;\;\; \forall u\in \Gamma(A),
\end{equation}
where $u^r$ is the vector field on $\grd$ determined by $u \in
\Gamma(A)$ via right translations; see \cite[Sec.~3]{bcwz} for the
proofs of these identities (the proofs there work in any degree,
though the statements refer to 2-forms). Using the second equation
in \eqref{eq:propm} and \eqref{eq:mu}, we also obtain
\begin{equation}\label{eq:mu2}
i_{\overline{u}^l}\omega = -\sour^* \mu(u),
\end{equation}
where $\overline{u}^l = \inv_*(u^r)$ (note that this vector field
coincides with the one defined by left translations of
$\overline{u}=d\inv (u) \in \Gamma(\ker(d\tar)|_M)$).

\begin{proposition}\label{prop:nondeg} A closed, multiplicative form
$\omega^{k+1}(\grd)$ is nondegenerate if and only if its
corresponding IM form $\mu: A \to \wedge^k T^*M$ satisfies
\begin{itemize}
\item[(1)] $\ker \mu = \{0\}$,
\item[(2)] $(\image(\mu))^\circ = \{X\in TM\,|\, i_X\mu(u)=0\;\forall u\in A\} = \{0\}$.
\end{itemize}
\end{proposition}

\begin{proof}

Assume that $\omega$ is nondegenerate, and let us verify that $(1)$
and $(2)$ hold. If $u\in \ker \mu$, then (by \eqref{eq:mu})
$i_u\omega =\tar^*\mu(u) =0$, so $u=0$ and $(1)$ follows. Let now
$X\in (\image(\mu))^\circ |_x$, $x\in M$. Then
$i_ui_X\omega=-i_X\tar^*\mu(u)=\tar^* i_X\mu(u)=0$ for all $u\in
A|_x$. We claim that this implies that $i_X\omega=0$, so that $X=0$
by nondegeneracy, and hence $(2)$ holds. To see that, it suffices to
check that $i_{Z_k} \ldots i_{Z_1}i_X\omega =0$ for arbitrary
$Z_i\in T\grd|_x$, $i=1,\ldots,k$. Using \eqref{eq:decomp}, we write
$Z_i = X_i + u_i$, for $X_i\in TM|_x$ and $u_i\in A|_x$. Expanding
out $i_{Z_k} \ldots i_{Z_1}i_X\omega$ using multilinearity, we see
that the term $i_{X_k} \ldots i_{X_1}i_X\omega$ vanishes by the
first condition in \eqref{eq:propm}, and all the other terms vanish
as a consequence of the fact that $i_ui_X\omega=0 \; \forall u\in
A$.


Conversely, suppose that $(1)$ and $(2)$ hold, and let $X\in
T_g\Grd$ be such that $i_X\omega=0$. Then
$$
i_{u^r}i_X\omega=0=-i_X(\tar^*\mu(u))
$$
for all $u \in \Gamma(A)$, which means that $d\tar(X)\in
(\image(\mu))^\circ$, so $d\tar(X)=0$ by $(2)$. Hence $X$ is tangent
to the $\tar$-fiber at $g$, and we can find $v\in \Gamma(A)$ so that
$\inv_*(v^r) |_g=\overline{v}^l |_g =X$. By \eqref{eq:mu2}, at the
point $g$ we have
$$
i_X\omega= i_{\overline{v}^l}\omega = -\sour^* \mu(v),
$$
so $i_X\omega=0$ implies that $\mu(v)=0$, hence $v=0$ by $(1)$, and
$X= \overline{v}^l |_g =0$.
\end{proof}

It follows that the infinitesimal counterpart of a $k$-plectic
groupoid is a closed IM $(k+1)$-form $\mu: A\to \wedge^k T^*M$
additionally satisfying conditions (1) and (2) of
Prop.~\ref{prop:nondeg}. A natural terminology for the resulting
object is {\em IM k-plectic form}. In this paper, we will
alternatively refer to them as  {\em higher Poisson structures of
degree $k$}, or simply {\em $k$-Poisson structures} (being aware
that this may clash with the terminology for different objects in
the literature). Before giving different characterizations of
$k$-Poisson structures and examples, we briefly explain how
1-Poisson structures are the same as ordinary Poisson structures.

\subsection{The case $k=1$}\label{subsec:k1}

For a bundle map $\mu: A\to T^*M$, note that condition $(1)$ in
Prop.~\ref{prop:nondeg} says that $\mu$ is injective, while $(2)$
says that $\mu$ is surjective. It follows that a 1-Poisson structure
is a bundle map $\mu: A\to T^*M$ satisfying \eqref{eq:IM1},
\eqref{eq:IM2} (i.e., a closed IM 2-form), and that is an
isomorphism.

Note that given a Poisson structure $\pi$ on $M$,  if we consider
the associated Lie algebroid $A=T^*M$, see \eqref{eq:anchor} and
\eqref{eq:liep}, it is clear that
\begin{equation}\label{eq:id}
\mu= \id: A \to T^*M
\end{equation}
is a 1-Poisson structure. It turns out that any 1-Poisson structure
is equivalent\footnote{We say that two IM $(k+1)$-forms $\mu_1:A_1
\to \wedge^kT^*M$ and $\mu_2:A_2 \to \wedge^kT^*M$ are {\em
equivalent} if there is a Lie-algebroid isomorphism $\phi: A_1\to
A_2$ such that $\mu_2\circ \phi = \mu_1$; these are infinitesimal
versions of isomorphism of Lie groupoids preserving multiplicative
forms.} to one of this type. To justify this claim, it will be
convenient to view Poisson structures from the broader perspective
of Dirac geometry \cite{courant}.

Let us consider the bundle $\TM := TM \oplus T^*M \to M$ equipped
with the non-degenerate, symmetric fibrewise bilinear pairing
$\SP{\cdot,\cdot}$ given at each $x\in M$ by
\begin{equation}\label{eq:pairing}
\SP{(X,\alpha),(Y,\beta)}:= \beta(X) + \alpha(Y),
\end{equation}
for $X,Y\in T_xM,\; \alpha,\beta \in T_x^*M$, and with the
Courant-Dorfman bracket $\Cour{\cdot,\cdot}: \Gamma(\TM)\times
\Gamma(\TM)\to \Gamma(\TM)$,
\begin{equation}\label{eq:courant}
\Cour{(X,\alpha),(Y,\beta)}:=([X,Y],\Lie_X\beta-i_Yd\alpha).
\end{equation}
Poisson structures on $M$ are equivalent to subbundles $L\subset
\TM$ satisfying
\begin{itemize}
\item[(d1)] $L=L^\perp$, i.e., $L$ is {\em lagrangian} with respect to $\SP{\cdot,\cdot}$,
\item[(d2)] $L\cap TM=\{0\}$,
\item[(d3)] $\Cour{\Gamma(L),\Gamma(L)}\subseteq
\Gamma(L)$.
\end{itemize}
Condition (d1) is equivalent to $L$ being isotropic, i.e., $L
\subseteq L^\perp$, and the dimension condition
$\mathrm{rank}(L)=\dim(M)$. Using the exact sequence
$$
L\cap TM \to L \to T^*M
$$
induced by the natural projection $\pr_2:\TM\to T^*M$, we see that
(d2) is equivalent to saying that $L$ projects isomorphically onto
$T^*M$. It follows that conditions (d1) and (d2) can be
alternatively written as
\begin{itemize}
\item[(d1')] $L \subseteq L^\perp$,
\item[(d2')] $\pr_{2}|_L: L \to T^*M$ is an
isomorphism.
\end{itemize}
Given a subbundle $L\subset \TM$, conditions (d1') and (d2') are
equivalent to $L$ being the graph of a skew-adjoint bundle map
$T^*M\to TM$; such maps are always of the form $\alpha\mapsto
i_\alpha\pi$, where $\pi$ is a bivector field. The involutivity
condition (d3) amounts to $[\pi,\pi]=0$.

Let $\mu: A\to T^*M$ be a 1-Poisson structure, and let us consider
the bundle map
\begin{equation}\label{eq:rmmap}
(\rho,\mu): A \to \TM,
\end{equation}
where $\rho: A\to TM$ is the anchor. Since $\mu$ is an isomorphism,
the map \eqref{eq:rmmap} is injective, and its image is a subbundle
$L\subset \TM$ satisfying (d2'). Note that condition \eqref{eq:IM1}
for $\mu$ amounts to condition (d1') for $L$, while \eqref{eq:IM2}
becomes (d3). It follows that $L$ represents a Poisson structure on
$M$, explicitly given by
$$
\pi(\alpha,\beta) = i_{\rho(\mu^{-1}(\alpha))}\beta,\qquad
\alpha,\beta\in T^*M.
$$
It is clear from \eqref{eq:IM2} that $\mu: A\to T^*M$ is an
isomorphism of Lie algebroids, where $T^*M$ has the Lie-algebroid
structure induced by $\pi$ (as in \eqref{eq:anchor} and
\eqref{eq:liep}), showing the equivalence between $\mu$ and the
1-Poisson structure \eqref{eq:id} associated with $\pi$.

As we see next, one has a similar interpretation of general
$k$-Poisson structures in terms of higher Courant-Dorfman brackets
(as in \cite[Sec.~2]{Hi}), leading to objects closely related to
those studied in \cite{Zambon}.

\section{Descriptions of $k$-Poisson structures}\label{sec:higherc}

Let us consider the vector bundle
$$
\TM^{(k)}:= TM\oplus \wedge^kT^*M;
$$
we denote by $\pr_1: \TM^{(k)} \to TM$ and $\pr_2: \TM^{(k)} \to
\wedge^kT^*M$ the natural projections. The same expressions as in
\eqref{eq:pairing} and \eqref{eq:courant} lead to a symmetric
$\wedge^{k-1}T^*M$-valued pairing $\SP{\cdot,\cdot}$ on the fibres
of $\TM^{(k)}$ and a bracket $\Cour{\cdot,\cdot}$ on
$\Gamma(\TM^{(k)})$, that we will keep referring to as the
Courant-Dorfman bracket.

Given a subbundle $L\subset \TM^{(k)}$, we keep denoting by
$L^\perp$ its orthogonal relative to $\SP{\cdot,\cdot}$; note that,
for $k>1$, it may happen that $L^\perp$ does not have constant rank
(see Section~\ref{sec:examples}). We will keep calling $L$ {\em
isotropic} if $L\subset L^\perp$, and {\em involutive} if its space
of sections $\Gamma(L)$ is closed under $\Cour{\cdot,\cdot}$. For a
subbundle $D\subseteq \wedge^k T^*M$, we let
$$
D^\circ:=\{X\in TM\,|\, i_X\alpha =0\, \forall \alpha\in D\}
$$
be its annihilator.

Whenever $L\subset \TM^{(k)}$ is an isotropic and involutive
subbundle, it inherits a Lie-algebroid structure with anchor map
$\pr_1|_L : L\to TM$ and Lie bracket
$\Cour{\cdot,\cdot}|_{\Gamma(L)}$ on $\Gamma(L)$. In particular, it
follows that the distribution
\begin{equation}\label{eq:dist}
\pr_1(L)\subseteq TM
\end{equation}
is integrable and its integral leaves (the ``orbits'' of the Lie
algebroid) define a singular foliation on $M$, see
\cite[Sec.~8.1]{DZ}. One may also directly check that
\begin{equation}\label{eq:IML}
\pr_2|_L : L \to \wedge^{k}T^*M
\end{equation}
is a closed IM $k$-form. Since $\ker(\pr_2|_L)= L\cap TM$ and
$$
(\pr_2(L))^\circ = L^\perp \cap TM \supseteq L\cap TM,
$$
it is clear that \eqref{eq:IML} is a $k$-Poisson structure if and
only if
\begin{equation}\label{eq:nondegL}
L^\perp\cap TM=\{0\}.
\end{equation}
By considering the bundle map \eqref{eq:IML}, we will think of any
isotropic, involutive subbundle $L\subseteq \TM^{(k)}$ satisfying
\eqref{eq:nondegL} as a $k$-Poisson structure. It turns out that all
$k$-Poisson structures on $M$ are of this type.

\begin{proposition}\label{prop:L}
Any $k$-Poisson structure $\mu: A\to \wedge^k T^*M$ is equivalent to
a subbundle $L\subset \TM^{(k)}$ that is isotropic, involutive, and
satisfies \eqref{eq:nondegL}.
\end{proposition}

\begin{proof}
Let $ \mu: A \to \wedge^k T^*M$ be a $k$-Poisson structure. The
bundle map $(\rho,\mu): A \to \TM$ is an isomorphism onto its image
(due to condition $(1)$ in Prop.~\ref{prop:nondeg}), which is a subbundle
$L\subset \TM^{(k)}$ that is isotropic, involutive, and satisfies
\eqref{eq:nondegL} (as a result of \eqref{eq:IM1}, \eqref{eq:IM2}
and condition $(2)$ in Prop.~\ref{prop:nondeg}, respectively). It is
clear that $(\rho,\mu): A \to L$ is an isomorphism of Lie
algebroids, which establishes the desired equivalence.
\end{proof}

We conclude that the infinitesimal versions of $k$-plectic groupoids
can be seen as isotropic, involutive subbundles $L\subset \TM^{(k)}$
satisfying \eqref{eq:nondegL}. Note that the condition $L=L^\perp$
(see (d1)) may not hold for $k> 1$ (we will see simple examples in
Section \ref{sec:examples}); in the case $k=1$, the condition
$L^\perp \cap TM = (\pr_2(L))^\circ =\{0\}$ implies that
$\pr_2(L)=T^*M$, so that $L=L^\perp$.

There is yet another characterization of $k$-Poisson structures,
closer in spirit to the description of Poisson structures via
bivector fields.

\begin{proposition}\label{prop:D}
There is a one-to-one correspondence between subbundles $L\subset
\TM^{(k)}$ as in Prop.~\ref{prop:L} and pairs $(D,\lambda)$, where
$D\subseteq \wedge^k T^*M$ is a subbundle and $\lambda: D\to TM$ is
a bundle map (covering the identity) satisfying the following
conditions: {(a)} $D^\circ =\{0\}$, {(b)} $i_{\lambda(\alpha)}\beta
= -i_{\lambda(\beta)}\alpha$, for $\alpha,\beta \in D$, and {(c)}
the space $\Gamma(D)$ is involutive with respect to the bracket
(c.f. \eqref{eq:liep})
\begin{equation}\label{eq:lbrk}
[\alpha,\beta]_\lambda := \Lie_{\lambda(\alpha)}\beta -
i_{\lambda(\beta)}d\alpha = \Lie_{\lambda(\alpha)}\beta -
\Lie_{\lambda(\beta)}\alpha - d(i_{\lambda(\alpha)}\beta),
\end{equation}
and $\lambda: \Gamma(D)\to \Gamma(TM)$ preserves brackets.
\end{proposition}

\begin{proof}
Given a $k$-Poisson structure $L\subset TM\oplus \wedge^k T^*M$,
note that $\pr_2|_L: L\to \wedge^k T^*M$ is injective (since
$\Ker(\pr_2|_L)=L\cap TM\subseteq L^\perp\cap TM=\{0\}$). Setting
$D=\pr_2(L)$ and $\lambda=\pr_1\circ (\pr_2|_L)^{-1}$, we see that
$L=\{(\lambda(\alpha),\alpha)\,|\, \alpha \in D\}$. Then
\eqref{eq:nondegL} is equivalent to condition $(a)$, while $(b)$
means that $L$ is isotropic. The involutivity of $L$ is equivalent
to condition $(c)$.
\end{proof}

For $k=1$, as previously remarked, $D=T^*M$ (as a result of $(a)$),
while $(b)$ says that $\lambda = \pi^\sharp$, for a bivector field
$\pi$. The involutivity condition in $(c)$ is automatically
satisfied, and the bracket-preserving property is equivalent to the
Poisson condition $[\pi,\pi]=0$ (see e.g. \cite[Lem.~2.3]{BC}).

For a $k$-Poisson structure defined by $(D,\lambda)$ as in
Prop.~\ref{prop:D}, $D$ acquires a Lie algebroid structure with
bracket \eqref{eq:lbrk} and anchor $\lambda$, in such a way that
$\pr_2|_L:L\to D$ is an isomorphism of Lie algebroids. In terms of
$(D,\lambda)$, the singular foliation on $M$ determined by the
$k$-Poisson structure (see \eqref{eq:dist}) is given by the integral
leaves of the distribution $\lambda(D)\subseteq TM$. Moreover, each
leaf $\mathcal{O}$ inherits a $(k+1)$-form $\omega$ by
\begin{equation}\label{eq:leafform}
\omega(Y_0,Y_1,\ldots,Y_k)= i_{Y_k}\ldots i_{Y_1}\alpha,
\end{equation}
where $Y_i\in \lambda(D)|_{\mathcal{O}}=T\mathcal{O}$, and $\alpha
\in D$ is such that $Y_0=\lambda(\alpha)$; indeed, property $(b)$ in
Prop.~\ref{prop:D} assures that $\omega$ is well defined. One may
also verify, using $(c)$ in Prop.~\ref{prop:D}, that $\omega$ is
closed. For $k=1$, one recovers the symplectic foliation that
underlies any Poisson structure and completely determines it.
However, for $k>1$, it is no longer true that the leafwise closed
$(k+1)$-forms are nondegenerate, nor that a $k$-Poisson structure is
uniquely determined by them, see Remark \ref{rem:fol2} (c.f.
\cite[Prop.~3.8]{Zambon}).

The description of $k$-Poisson structures in Prop.~\ref{prop:D} also
makes the notion of morphism of $k$-Poisson manifolds more evident:
if $(D_i,\lambda_i)$ is a $k$-Poisson structure on $M_i$, $i=1,2$,
then a map $\phi: M_1\to M_2$ is a {\em $k$-Poisson morphism} if,
for all $x\in M_1$, $\phi^*(D_2|_{\phi(x)})\subseteq D_1|_x$ and
$d\phi(\lambda_1(\phi^*\alpha))=\lambda_2(\alpha)$, for all
$\alpha\in D_2|_{\phi(x)}$.

\section{Some examples and final remarks}\label{sec:examples}

We now give some examples of $k$-Poisson structures. The first two
examples are from \cite{Zambon}.

\begin{example}\label{ex:multisymp}
Let $\omega\in \Omega^{k+1}(M)$ be a $k$-plectic form. Then its
graph
$$
L=\{(X,i_X\omega),\; X\in TM\} \subset \TM^{(k)}
$$
satisfies $L=L^\perp$ and is involutive (as a consequence of
$\omega$ being closed, see \cite[Prop.~3.2]{Zambon}). Also,
$L^\perp\cap TM = L\cap TM = \ker(\omega)=\{0\}$ by nondegeneracy.
In terms of Prop.~\ref{prop:D}, $D=\mathrm{Im}(\omega^\sharp)$ and
$\lambda= (\omega^\sharp)^{-1}:D\to TM$. So, just as any symplectic
structure is a Poisson structure, any $k$-plectic form is a
particular type of $k$-Poisson structure. A $k$-plectic groupoid
integrating this $k$-Poisson structure is the pair groupoid $M\times
M$, with $k$-plectic structure $p_1^*\omega - p_2^*\omega$ where
$p_i$, $i=1,2$, denote the two natural projections from $M\times M$
to $M$.
\end{example}

Considering a $k$-plectic groupoid $\grd\toto M$ with the
$k$-Poisson structure of Example~\ref{ex:multisymp}, one may use
\eqref{eq:mu} to check that the target map $\tar: \grd\to M$ is a
$k$-Poisson morphism, extending the well-known property of
symplectic groupoids, see Section~\ref{sec:poisson}.

We saw in Section~\ref{subsec:k1} that Poisson bivector fields are
the same as 1-Poisson structures. Other types of higher Poisson
structures are obtained from top-degree multivector fields as
follows.

\begin{example}\label{ex:multivector}
Let $\pi \in \Gamma(\wedge^{k+1}TM)$ be a multivector field of top
degree, i.e., $k=\dim(M)-1$. Then its graph
$$
L=\{(i_\alpha\pi,\alpha)\;|\; \alpha\in \wedge^kT^*M\}\subseteq
\TM^{(k)}
$$
is isotropic and involutive -- and, besides Poisson bivector fields,
these are the only examples of non-zero multivector fields whose
graphs have these properties, see \cite[Prop.~3.4]{Zambon}. Also,
since $\pr_2(L)=\wedge^k T^*M$, it is clear that $\pr_2(L)^\circ =
L^\perp\cap TM =\{0\}$, so $L$ is a $k$-Poisson structure. The
foliations defined by these $k$-Poisson structures are usually
singular: leaves are either open subsets of $M$ or singular points
(where $\pi$ vanishes). The restriction of $\pi$ to each open leaf
is nondegenerate, and the induced $(k+1)$-forms $\omega$ on these
leaves (see \eqref{eq:leafform}) are the volume forms dual to $\pi$,
i.e., they are defined by $i_{(i_\alpha \pi)} \omega = \alpha,
\forall \alpha \in \wedge^kT_x^*M$. The groupoids integrating these
$k$-Poisson structures have been mostly studied when $\dim(M)=2$ (so
$\pi$ is a bivector field), see \cite{GL,Mart}.

\end{example}

The fact that the particular $k$-Poisson structures of
Examples~\ref{ex:multisymp} and \ref{ex:multivector} are
infinitesimal versions of $k$-plectic groupoids was observed in
\cite[Prop.~3.7]{Zambon}.

In the preceding examples, the bundle $L$ always satisfied
$L=L^\perp$. For examples where this condition fails, consider
subbundles
\begin{equation}\label{eq:Lk}
L\subseteq \wedge^kT^*M \subset \TM^{(k)}.
\end{equation}
These are automatically isotropic and involutive. Note that
$$
L^\perp
= L^\circ \oplus \wedge^k T^*M, \;\; \mbox{ and }\;\;
L^\perp \cap TM = L^\circ.
$$
So $L$ is a $k$-Poisson structure as long as $L^\circ=\{0\}$, and $L
\subsetneq L^\perp$ as long as $L$ is properly contained in
$\wedge^kT^*M$.  A $k$-plectic groupoid integrating it is $L$
itself, viewed as a vector bundle (with groupoid structure given by
fibrewise addition), equipped with the $k$-plectic form given by the
pullback of the canonical multisymplectic form on $\wedge^kT^*M$
(see \eqref{eq:can}); the fact that this pullback is nondegenerate
boils down to the condition $L^\perp \cap TM=L^\circ = \{0\}$.

\begin{example}\label{ex:T*k}
For $L=\wedge^k T^*M$, note that $L^\circ=\{0\}$ (and hence $L$ is a
$k$-Poisson structure on $M$) if and only if $\dim(M)\geq k$.
\end{example}

\begin{example}
Let $\xi$ be a nondegenerate $k$-form on $M$, and let
$L\subset \wedge^kT^*M$ be the line bundle generated by $\xi$,
$$
L|_x=\{c\xi_x \; | \; c\in \mathbb{R}\},\;\;\; x\in M.
$$
Then $L^\circ = \ker(\xi)=0$, so $L$
is a $k$-Poisson structure.

\end{example}

\begin{remark}\label{rem:fol2}
Note that all $k$-Poisson structures of the type \eqref{eq:Lk}
determine the same foliation, the leaves of which are the points of
$M$.
\end{remark}

A general observation is that one can take direct products of
$k$-Poisson structures: if $L_1$ and $L_2$ are $k$-Poisson
structures on $M_1$ and $M_2$, respectively, we define their product
by
$$
L := \{ (X+ Y, \alpha + \beta)\;|\; (X,\alpha)\in L_1,\,
(Y,\beta)\in L_2 \} \subseteq TM\oplus \wedge^kT^*M,
$$
where $M=M_1\times M_2$ and we simplify the notation by identifying
forms on $M_i$ with their pullbacks to $M$ via the projections. One
may directly verify that $L$ is a $k$-Poisson structure on $M$.
Moreover, if $(\grd_i\toto M_i,\omega_i)$ is a $k$-symplectic
groupoid integrating $L_i$, $i=1,2$, the direct product
$\grd_1\times \grd_2\toto M_1\times M_2$ (equipped with the
$k$-plectic form $\omega_1 + \omega_2$) is a $k$-plectic groupoid
that integrates $L$. The following is a concrete example.

\begin{example}\label{ex:prod}
Let $(M,\omega)$ be a $k$-plectic manifold, and let $N$ be a
manifold with $\dim(N)\geq k$. Then the subbundle
$$
L=\{(X, i_X\omega +\alpha)\;|\; X\in TM, \alpha\in \wedge^k T^*N\}
\subset T(M\times N)\oplus \wedge^k T^*(M\times N)
$$
is a $k$-Poisson structure on $M\times N$ (c.f.
\cite[Thm.~3.12]{Zambon}), the direct product of the $k$-plectic
form on $M$ with the $k$-Poisson structure $L=\wedge^kT^*N$ on $N$
(see Example~\ref{ex:T*k}). The leaves of $L$ are $M\times \{t\}$,
$t\in N$, with induced $(k+1)$-form (as in \eqref{eq:leafform})
given by $\omega$.
\end{example}

The next observation illustrates that $k$-Poisson structures become
more rigid than Poisson structures when $k>1$.

\begin{remark} Let $M$ and $N$ be as is Example~\ref{ex:prod}, let $f\in
C^\infty(N)$, and consider the smooth family $\omega_t =
f(t)\omega$, $t\in N$, of $k$-plectic forms on $M$. For $k=1$, this
family defines a Poisson structure on $M\times N$, uniquely
determined by the fact that its symplectic leaves are $(M\times
\{t\},\omega_t)$. A higher generalization of this Poisson structure
is given by the (isotropic) subbundle $L\subset T(M\times N)\oplus
\wedge^k T^*(M\times N)$ defined by
$$
L|_{(x,t)}=\{(X, i_X\omega_t +\alpha)\;|\; X\in T_xM, \alpha\in
\wedge^k T_t^*N\}.
$$
As it turns out, for $k>1$, one may verify that such $L$ is
involutive if and only if $df=0$, i.e., $f$ is (locally) constant.
\end{remark}


We finally mention another product-type operation for
multisymplectic manifolds leading to higher Poisson structures that
are not multisymplectic.

\begin{example}
Let $(M_i,\omega_i)$ be a  $k_i$-plectic manifold, $i=1,2$. Let
$M=M_1\times M_2$ and $\omega = \omega_1 \wedge \omega_2 \in
\Omega^{k_1+k_2 + 2}(M)$ (we keep the simplified notation of
identifying forms on $M_i$ with their pullbacks to $M$ via the
projections $M\to M_i$). Then
$$
L =\{(X,i_X\omega)=(X,(i_X\omega_1)\wedge \omega_2) \;|\; X\in
TM_1\} \subset TM\oplus \wedge^{k_1+k_2+1}T^*M
$$
can be checked to be a $(k_1+k_2+1)$-Poisson structure. Its leaves
are of the form $M_1\times \{y\}$, for $y\in M_2$, and the induced
$(k_1+k_2+2)$-form on each leaf is zero. An integrating $k$-plectic
groupoid is given by the direct product of the pair groupoid $M_1
\times M_1$ (see Example \ref{ex:multisymp}) and the trivial
groupoid over $M_2$, endowed with the multiplicative $(k_1+k_2 +
2)$-form given by $(p_1^*\omega_1 - p_2^*\omega_1)\wedge \omega_2$.

\end{example}


\end{document}